\documentclass[11pt,a4paper,twoside]{article}
\usepackage[backend=bibtex, 
style=alphabetic,
sortcites=true,
doi=false,isbn=false,url=false,
maxbibnames=99,
sorting=nyt,
]{biblatex}

\bibliography{Bibliography}{}

\usepackage[T1]{fontenc}
\usepackage[utf8]{inputenc}
\usepackage[english]{babel}
\usepackage{amssymb, amsmath, amsfonts, amsthm}
\usepackage{enumerate}
\usepackage{colonequals}
\usepackage{empheq,fancybox}
\usepackage{color}
\usepackage[noblocks]{authblk}

\usepackage{microtype}
\usepackage{ellipsis}
\usepackage[BCOR=20mm,DIV=11]{typearea}
 \newcommand{\hm}[1]{\leavevmode{\marginpar{\tiny%
 $ \hbox to 0mm{\hspace*{-0.5mm} $ \leftarrow $ \hss}%
 \vcenter{\vrule depth 0.1mm height 0.1mm width \the\marginparwidth}%
 \hbox to
 0mm{\hss $ \rightarrow $ \hspace*{-0.5mm}} $ \\\relax\raggedright #1}}}

\newcommand{\eps}{\varepsilon}

\newcommand{\R}{\mathbb{R}}

\newcommand{\cl}{\operatorname{cl}}

\newcommand{\tvert}[1]{{\left\vert\kern-0.25ex\left\vert\kern-0.25ex\left\vert #1
    \right\vert\kern-0.25ex\right\vert\kern-0.25ex\right\vert}}
\renewcommand{\epsilon}{\varepsilon}

\DeclareMathOperator{\Deg}{Deg}
\newtheorem{theorem}{Theorem}[section]
\newtheorem{lemma}[theorem]{Lemma}
\newtheorem{proposition}[theorem]{Proposition}

\theoremstyle{definition}

\theoremstyle{remark}

\begin{document}
 \title{Perpetual cutoff method and discrete Ricci curvature bounds with exceptions}
 \author{Florentin M\"unch}
 \date{\today}
 \maketitle
 %
 %
 %
 %
 \begin{abstract}
One of the main obstacles regarding Barky Emery curvature on graphs is that the results require a global uniform lower curvature bounds where no exception sets are allowed.
We overcome this obstacle by introducing the
perpetual cutoff method.
As applications, we prove gradient estimates only requiring curvature bounds on parts of the graph. 
Moreover, we sharply  upper bound the distance to the exception set for graphs having uniformly positive Bakry Emery curvature everywhere but on the exception set. 
 \end{abstract}

\section{Introduction}
Recently, there has been tremendous interest in Ricci curvature on graphs (see 
\cite{Sch99, LinYau10,LLY11,BHLLMY15,  Mun17,johnson2015discrete,kempton2017relationships,
fathi2018curvature,liu2017rigidity,
liu2017distance,GongLin15,
liu2016bakry,cushing2016bakry,chung2014harnack,
yamada2017curvature,gao2016one,HornLinLiuYau14}
for Bakry-Emery Ricci curvature; see
\cite{BJL12,JL14,ni2015ricci,Oll09,Ollivier2007,
ollivier2012curved,eldan2017transport,
bhattacharya2015exact,paeng2012volume,
sandhu2015graph,wang2014wireless,rubleva2016ricci,MuenchWojciechowski17}
for Ollivier Ricci curvature, and see
\cite{maas2017entropic,fathi2016entropic,
erbar2012ricci,najman2017modern,erbar2016poincare,
EMT14}
for entropic Ricci curvature).

An important result for Bakry Emery curvature on graphs is the gradient estimate 
\[
\Gamma P_t f \leq e^{-2Kt}P_t \Gamma  f,
\]
for the heat semigroup $P_t = e^{t\Delta}$, assuming the curvature dimension condition $CD(K,\infty)$ everywhere, see \cite{HuaLin17,GongLin15,LinLiu15,LiuP14}. As soon as one vertex of the graph behaves badly and has negative curvature, the gradient estimate potentially gets far off from what would expect. For large girth graphs, a geometric cutoff might be possible, i.e., one can hope that induced subgraphs inherit lower curvature bounds. However this appears to be hopeless for small girth graphs since cutting out edges will possibly reduce connectedness within one-spheres which usually decreases curvature.

We overcome the obstacle of cutting off an exception set $W$ by introducing an analytic cutoff method inspired by the gradient estimate proofs in \cite{MuenchWojciechowski17}.
The idea is to define a modified semigroup $P_t^W$ staying constant in space (but not in time) on the exception set $W$, and behaving like the standard semigroup on the part with curvature bound.
Staying constant on the exception set makes the curvature there invisible since it cannot contribute to the gradient. Behaving like the standard semigroup on the part with curvature bound allows to employ the $\Gamma$-calculus yielding gradient estimates only requiring curvature bounds on parts of the graph, see Section~\ref{sec:GradEstimates}.
In particular in Theorem~\ref{thm:gradEstimate}, we prove
\[
\Gamma P_t^W f \leq e^{-2Kt} P_t f
\]
when assuming $CD(K,\infty)$ only on $V\setminus W$.
As application, we give distance bounds under non-constant Bakry-Emery curvature, see Section~\ref{sec:DistanceBounds}.
In particular, we obtain sharp distance bounds recovering the Bonnet-Myers diameter bounds from \cite{liu2016bakry}. In Theorem~\ref{thm:DistanceBounds}, we will prove
\begin{align*}
d(\cdot, W) \leq \frac {2D}K +1
\end{align*}
if the graph satisfies $CD(K,\infty)$ on $V \setminus W$ for some $K>0$. Here, $D$ denotes the  maximal vertex degree.
This greatly improves a result of \cite{liu2017distance} where it was shown
\[
d(\cdot,W) \leq Ce^{4K_0/K}\frac {D}{\sqrt{KK_0}}  + 1
\]
where an overall curvature bound $Ric \geq -K_0$ is needed, i.e., the distance bound can get arbitrarily far off if the curvature on $W$ behaves badly.
In contrast, our result does not depend on the the curvature of $W$ at all and gives the sharp distance bound one would expect from the Bonnet Myers theorem on graphs (see \cite[Corollary~2.2]{liu2016bakry}).
Finally in Section~\ref{sec:Discussion}, we give a general discussion on the significance of the new cutoff method and point out further research directions for possible interest.

\section{Setup and notation}

We say $G=(V,w,m)$ is a \emph{graph} if $V$ is a countable set called vertex set, if $w:V^2 \to [0,\infty)$ is  symmetric and zero on the diagonal, and if $m:V\to (0,\infty)$. The function $w$ is called \emph{edge weight} and $m$ is called \emph{vertex measure}.
We write $x\sim y$ for $x,y \in V$ if $w(x,y)>0$.
We will only consider \emph{locally finite} graphs, i.e., $\#\{y \in V: y\sim x\} < \infty$ for all $x \in V$.
The (weighted) vertex degree is given by
$\Deg(x) := \sum_y w(x,y)/m(x)$ for $x \in V$.
We will assume that $\Deg$ is bounded as a standing assumption in the paper.
The combinatorial graph distance $d$ on $V$ is given by $d(x,y) := \inf \{n: x=x_0 \sim \ldots \sim x_n=y\mbox{ for some } x_k \in V \}$. We will always assume that $G$ is connected, i.e., $d(x,y)<\infty$ for all $x,y\in V$.
We define the function spaces $C(V):=\R^V$ and $\ell_\infty(V):=\{f \in C(V): \|f\|_\infty < \infty\}$ with $\|f\|_\infty:= \sup_{x\in V} |f(x)| $.
The \emph{graph Laplace operator} $\Delta: C(V) \to C(V)$ is defined by
\[
\Delta f(x):= \frac{1}{m(x)}\sum w(x,y)(f(y) - f(x)).
\]
Denote by $P_t=e^{t\Delta}$ the associated heat semigroup.
We now introduce the Bakry-Emery calculus following \cite{Sch99,LinYau10,BakryGentilLedoux} to define Ricci curvature on graphs.
The gradient form $\Gamma: C(V)^2 \to C(V)$ is given by
\[2\Gamma (f,g) := \Delta(fg) - f\Delta g -g \Delta f  \]
and $\Gamma_2: C(V)^2 \to C(V)$ is given by
\[2\Gamma_2 (f,g) := \Delta\Gamma(f,g) - \Gamma(f,\Delta g) -\Gamma(g, \Delta f).  \]
For convenience, we write $\Gamma f:= \Gamma(f,f)$ and $\Gamma_2 f := \Gamma_2(f,f)$.

We say $G$ satisfies the \emph{curvature dimension condition} $CD(K,n)$ at $x \in V$ if for all $f \in C(V)$,
\[
\Gamma_2 f(x) \geq \frac 1 n (\Delta f)^2(x) + K\Gamma f (x).
\]
This is analog to Bochner's inequality on manifolds and defines the Bakry-Emery curvature on graphs where $K\in \R$ is a lower curvature bound and $n >0$ is an upper dimension bound. 
We say $G$ satisfies $CD(K,n)$ if it is satisfied at all $x\in V$.
The $CD(K,n)$ inequality was initially used as a curvature bound on diffusion semigroups and manifolds in \cite{Bakry87,BakryEmery85}, and later transfered to discrete settings in \cite{Sch99,LinYau10}.
For the sake of clarity, we restrict ourselves to the case $n=\infty$ where the maximum vertex degree in some sense substitutes the notion of dimension.
Note that in case of bounded vertex degree $\Deg \leq D$, one can always trade off the dimension term with a curvature term. In particular,
\[
CD(K,n) \Rightarrow CD(K,\infty) \Rightarrow CD\left(K-\frac{2D} n,n\right)
\]
since $(\Delta f)^2 \leq 2D\Gamma f$ due to Cauchy-Schwarz.

\section{Dini Derivatives, Gamma calculus, and semigroups}

For dealing with the non-smoothening effects of the cutoff, we define the upper and lower left and right \emph{Dini derivatives} for functions $F : I \to \R$  for an interval $I$ via
\[
\overline{\partial_t^{\pm}} F(t) := \limsup_{h \to 0^\pm} \frac{F(t+h) - F(t)}{h}  
\]
and
\[
\underline{\partial_t^{\pm}} F(t) := \liminf_{h \to 0^\pm} \frac{F(t+h) - F(t)}{h}.  
\]
If the limits exist, we write $\partial_t^\pm$ for the left and right derivative.
We next give basic properties of the Dini derivative from
\cite{hagood2006recovering}.
\begin{proposition}[Monotonicity theorem and Theorem~3 from \cite{hagood2006recovering}]
\label{pro:Dini}
Let $F: I \to \R$ be continuous. 
\begin{enumerate}[(i)]
\item
If $\overline{\partial_t^+} F \geq 0$ on the interior of $I$, then $F$ is non-decreasing.
\label{i:Dininonnegative}
\item If $|\partial_t^+ F| < \infty$ everywhere, then
\[
F(b)-F(a) \leq \overline{\int_a^b} \partial_t^+ F(t) dt
\]
on every interval $[a,b]$ where $\overline \int$ is the upper Riemann integral.
\label{i:DiniIntegral}
\end{enumerate}
\end{proposition}
As a preparation for the cutoff method, we show compatibility of the right derivatives, the semigroup, and the $\Gamma$ calculus.

\begin{lemma}\label{lem:DiniPtGamma}
 Let $G=(V,w,m)$ be a graph with bounded vertex degree.
Let $(u_t)_{t\geq 0} \in \ell_\infty (V)$ be continuous in $t$ w.r.t $\|\cdot\|_\infty$. 
Suppose the right derivative $\partial_t^+ u_t(x)$ exists everywhere. Then for $t>s\geq 0$,
\begin{enumerate}[(i)]
\item $\partial_t^+ \Gamma u_t = 2\Gamma(u_t,\partial_t^+ u_t)$.
\label{i:dtGammaU}
\item $\partial_s^+ (P_{t-s} u_s) =  P_{t-s}\left(\partial_s^+ - \Delta \right) u_s$.
\label{i:dtPtU}
\end{enumerate}
\end{lemma}
\begin{proof}
Claim $(\ref{i:dtGammaU})$ follows from the interchangeability of $\partial_t^+$ and $\Delta$, and from the product rule $\partial_t^+(fg) = g\partial_t^+ f + f \partial_t^+ g$.

We now prove claim $(\ref{i:dtPtU})$. Let $x \in V.$ Observe
\begin{align*}
\frac{P_{t-s-\eps}u_{s+\eps} - P_{t-s}u_s}\eps(x) =
\frac{P_{t-s-\eps}u_{s+\eps} - P_{t-s}u_{s+\eps}}\eps (x) + \frac{P_{t-s}(u_{s+\eps} -u_s)}\eps(x).
\end{align*}
Due to continuity of $P_t$, the latter part converges to $P_{t-s} \partial_s^+ u_s(x)$.

Due to the mean value theorem, there exists $\delta=\delta(\eps) \in [0,\eps]$ s.t. the former part equals $-\Delta P_{t-s-\delta} u_{s+\eps}(x)$ 
which, due to continuity of $u_t$ and $P_t$, converges to $-\Delta P_{t-s}u_s(x)$ as $\eps$ and thus $\delta(\eps)$ tend to zero.
Putting together, and applying commutativity of $\Delta$ and $P_t$ on $\ell_\infty(V)$ following from bounded vertex degree, gives $\partial_s^+ (P_{t-s} u_s) = P_{t-s}(\partial_s^+ -\Delta)u_s$ as desired. 
\end{proof}

\section{Perpetual cutoff method}\label{sec:Cutoff}

The perpetual cutoff method we introduce in this section is inspired by the cutoff method given in \cite{MuenchWojciechowski17} to prove gradient estimates for Ollivier curvature. The main difference is that in \cite{MuenchWojciechowski17}, the cutoff threshold depends on space, but not on time, whereas in this article, the cutoff threshold only depends on time, but not on space.

Let $\emptyset \neq W \subset V$.
We define the \emph{closure} 
\[\cl(W) := \{v\in V: d(v,W) \leq 1\}\]
and we define \[\ell_\infty^W(V) := \{f \in \ell_\infty(V): f \geq  f(w) \mbox{ for all } w \in \cl(W)\}.\]
This is the set of all functions being constant on $\cl(W) $ and not smaller outside.
We define the cutoff operator $S^W:\ell_\infty(V) \to \ell_\infty^W(V)$ via
\[
S^W f := f \vee \sup_{\cl(W)} f.
\]
The reason why we define $S^W$ via the closure $\cl(W)$ is that we want the gradient $\Gamma f$ to be zero on $W$ for $f \in \ell_\infty^W(V)$. This turns out to be essential for showing that the gradient estimate is independent of the curvature  on $W$. 
We next define the single time application of the cutoff
$Q_t^W: \ell_\infty^W(V) \to \ell_\infty^W(V)$ via
\[
Q_t^W f := S^W P_t f. 
\]
We finally come to the most relevant definition of the paper.
The (non-linear) perpetual cutoff semigroup $P_t^W : \ell_\infty^W(V) \to \ell_\infty^W(V)$ is given by
\begin{align*}
P_t^W f:= \sup_{t_1 + \ldots +  t_n = t} Q_{t_1}^W \ldots Q_{t_n}^W f
\end{align*}
where $n$ is arbitrary all $t_i$ are positive.
We are now prepared to give a collection of useful properties of $P_t^W$ and its domain $\ell_\infty^W(V)$.
\begin{theorem}\label{thm:PtWProperties}
Let $G=(V,w,m)$ be a graph with bounded vertex degree $\Deg \leq D$. Let $\emptyset \neq W\subset V$.
Let $f \in \ell_\infty^W(V)$. Then for $x\in V$ and $s,t\geq 0$,
\begin{enumerate}[(i)]
\item $P_s^W \circ P_t^W = P_{s+t}^W$. \label{i:PtWSemigroup}
\item
$\|P_t^W f\|_\infty \leq \|f\|_\infty$.
\label{i:contractionSemigroup}
\item
$\|P_t^W f - f\|_\infty \leq 2tD\|f\|_\infty$.
\label{i:Continuity}
\item
$\partial_t^+P_t^W f (x)|_{t=0} = \partial_t^+Q_t^W f (x)|_{t=0}= 
\begin{cases} S^W \Delta f(x)&: f(x)= \inf_V f \\ \Delta  f(x) &: \mbox{otherwise}. 
\end{cases}$ \label{i:QtWPtW}
\item $\|\partial_t^+ P_t^W f\|_\infty \leq \| \Delta P_t^W f\|_\infty$.
\label{i:dtPtW}
\item
$\partial_t^+ \Gamma P_t^W f \leq 2\Gamma(P_t^W f, \Delta P_t^W f)$.\label{i:dtGammaPtW}
\item $\Gamma_2 f \geq 0$ and $\Gamma f = 0$ on $W$.
\label{i:Gamma2PtW}
\item $P_\infty^W f := \lim_{t\to \infty} P_t^W f$ is constant on $V$. \label{i:PinftyW}
\end{enumerate}

\end{theorem}

\begin{proof}
We first prove the semigroup property $(\ref{i:PtWSemigroup})$.
Let $S=\{0<\sigma_1 < \ldots < \sigma_m=s\}$ be a partition of $[0,s]$ and $T=\{0<\tau_1 < \ldots < \tau_n=t\}$ a partition of $[0,t]$. We canonically define $Q_S^W := Q_{\sigma_m - \sigma_{m-1}}^W\ldots Q_{\sigma_1}^W$ and $Q_T^W := Q_{\tau_n - \tau_{n-1}}^W\ldots Q_{\tau_1}^W$. Note that $Q_{T'}^W f \geq Q_T^W f$ if $T'$ is a refinement of $T$.
Obviously, $P_t^W f = \sup_T Q_T^W$ where the supremum is taken over all partitions of $[0,t]$.
Observe
\[ 
Q_S^W Q_T^W f = Q_{(t+S) \cup T}^W f \leq P_{t+s}^W
\]
and therefore, $P_s^WP_t^W f \leq P_{t+s}^W f$.
On the other hand let $U$ be partition of $[0,s+t]$ and $U' := U \cup \{t\}$ a refinement.
Let $S:= (U-t) \cap(0,s]$ and $T=U'\cap (0,t]$. Then,
\[
P_s^W P_t^W f \geq  Q_S^W Q_T^W f = Q_{U'}^W f \geq Q_U^W f.
\]
Taking supremum yields $P_s^WP_t^W f \geq P_{t+s}^W f$ which proves $(\ref{i:PtWSemigroup})$.

Claim $(\ref{i:contractionSemigroup})$ is clear as $\inf_V f \leq \inf_V Q_t^W f \leq  \sup_V Q_t^W f \leq \sup_V f$.

For claim $(\ref{i:Continuity})$ observe
$|P_t f - f| \leq 2tD\|f\|_\infty$ giving $|Q_t^W f(x) - f(x)| \leq 2tD\|f\|_\infty$ for $x \in V$ s.t. $P_t f(x)> \sup_{\cl(W)} P_t f$.
Moreover since $f \in \ell_\infty^W(V)$,
\[
\inf_V Q_t^W f - \inf_V f = \sup_{\cl(W)} P_t f - \sup_{\cl(W)} f \leq 2tD\|f\|_\infty. 
\]
Therefore if $P_t f(x) \leq \sup_{\cl(W)} P_t f$,
we have
\[
- 2tD\|f\|_\infty \leq  P_t f(x) - f(x) \leq Q_t^W f(x) - f(x) \leq \inf_V Q_t^W f - \inf_V f \leq 2tD\|f\|_\infty. 
\]
implying $\|Q_t^W f - f\|_\infty \leq 2tD\|f\|_\infty$.
Due to $(\ref{i:contractionSemigroup})$, this gives
$\|P_t^W f - f\|_\infty \leq 2tD\|f\|_\infty$
which proves claim $(\ref{i:Continuity})$.

We now prove $(\ref{i:QtWPtW})$. 
First note that $\inf_V Q_t^W f = \sup_{\cl(W)} P_t f$ and $\inf_V f = \sup_{\cl(W)} f$ since $f \in \ell_\infty^W(V)$.
Due to the mean value theorem and since $\partial_s P_s f = \Delta P_s f$,
\begin{align*}
t \sup_{\cl(W)} \inf_{[0,t]} \Delta P_s f\leq \inf_V Q_t^W f - \inf_V f \leq t \sup_{[0,t]\times \cl(W)} \Delta P_s f.
\end{align*}
Moreover, $\|\Delta P_s f - \Delta f\|_\infty \leq 2D \|P_s f - f\|_\infty \leq  4s D^2 \|f\|_\infty$ and thus,
\begin{align}
\left|\inf_V Q_t^W f - \inf_V f - t\sup_{\cl(W)} \Delta f \right|\leq  4t^2D^2\|f\|_\infty. \label{eq:infQtW}
\end{align}
Since $\inf_V Q_t^W f = \sup_{\cl(W)} P_t f$ and by the mean value theorem, there exists $s \in [0,t]$ s.t.
\[
Q_t^W f = (f + t \Delta P_s f) \vee \inf_V Q_t^W f.
\]
Putting together with \eqref{eq:infQtW} and  applying $\|\Delta P_s f - \Delta f\|_\infty \leq 4s D^2 \|f\|_\infty$ again yields
\[
\left| Q_t^W f - (f+t\Delta f) \vee (t \sup_{\cl(W)} \Delta f + \inf_V f) \right| \leq  4t^2D^2\|f\|_\infty.
\]
In particular, if $f(x)=\inf_V f$,
\[
\left| Q_t^W f(x) - f(x)- t S^W \Delta f(x) \right| \leq  4t^2D^2\|f\|_\infty.
\]
If in contrast $f(x)>\inf_V f$, then for small $t>0$, one has $f+ t\Delta f \geq t \sup_{\cl(W)} \Delta f + \inf_V f$ yielding
\[
\left| Q_t^W f(x)-f(x) - t \Delta f(x) \right| \leq  4t^2D^2\|f\|_\infty.
\]
Putting together yields
\[
\partial_t^+ Q_t^W f(x)|_{t=0} = L^W f(x) :=
\begin{cases} S^W \Delta f(x)&: f(x)= \inf_V f \\ \Delta  f(x) &: \mbox{otherwise} 
\end{cases}
\]
and more precisely, for small enough $t>0$,
\begin{align}
\left| Q_t^W f(x)-f(x) - t L^W f(x) \right| \leq  4t^2D^2\|f\|_\infty.
\label{eq:QtWassymp}
\end{align}
Note that $L^W f(x)$ is upper semi-continuous in $f$ w.r.t. $\|\cdot\|_\infty$.
Let $t_1 + \ldots + t_n = t$ and let $f_n:=Q_{t_{n-1}}^W \ldots Q_{t_1}^W f$.
Applying \eqref{eq:QtWassymp} iteratively yields
\begin{align}
\left| 
Q_{t_n}^W\ldots Q_{t_1}^W f(x) - f(x)  - \sum_{k=1}^n t_k L^W f_k(x)
\right|
\leq 4D^2\|f\|_\infty \sum_{k=1}^n t_k^2 \leq 4t^2D^2\|f\|_\infty.
\label{eq:QtnLwfk}
\end{align}
Since $\|f_k - f\|_\infty \leq 2tD\|f\|_\infty$ and due to upper semi-continuity of $L^W f(x)$, we have
$L^W f_k(x) \leq L^W f(x) + \eps(t)$ with $\lim_{t\to 0}\eps(t)=0$. Thus, \eqref{eq:QtnLwfk} gives
\begin{align*}
Q_{t_n}^W\ldots Q_{t_1}^W f(x) - f(x) \leq 4t^2D^2\|f\|_\infty + t(L^W f(x) + \eps(t)).
\end{align*}
Taking supremum yields
$\overline{\partial_t^+} P_t^W f|_{t=0} \leq L^W f$.
On the other hand, $P_t^W f \geq Q_t^W f$ and thus,
$\underline{\partial_t^+} P_t^W f|_{t=0} \geq \partial_t^+ Q_t^W f|_{t=0}=L^W f$. This proves $(\ref{i:QtWPtW})$.

Claim $(\ref{i:dtPtW})$ follows immediately from $(\ref{i:PtWSemigroup})$ and $(\ref{i:QtWPtW})$.

We next prove $(\ref{i:dtGammaPtW})$.
Due to the semigroup property, it suffices to prove
$\partial_t^+ \Gamma P_t^W f \leq 2\Gamma(P_t^W f, \Delta P_t^W f)$ at $t=0$.
First observe
$\partial_t \Gamma P_t f = 2 \Gamma(P_t f, \Delta P_t f)$ and $\Gamma S^W f \leq \Gamma f$ for all $f \in \ell_\infty(V)$ since $S^W$ is a contraction. Due to the mean value theorem, there exists $0\leq s\leq t$ s.t.
\begin{align*}
\Gamma Q_t^W f - \Gamma f \leq \Gamma P_t f - \Gamma f = 2t\Gamma(P_s f, \Delta P_s f)
\end{align*}
implying $\overline{\partial_t^+} \Gamma Q_t^W f |_{t=0} \leq 2 \Gamma(f,\Delta f)$ due to continuity of $\Gamma$, $\Delta$ and $P_t$.
Due to $(\ref{i:QtWPtW})$, we have $\partial_t^+Q_t^Wf|_{t=0} = \partial_t^+P_t^W f |_{t=0}$.
Thus by Lemma~\ref{lem:DiniPtGamma} $(\ref{i:dtGammaU})$,
\begin{align*}
\partial_t^+ \Gamma Q_t^W f|_{t=0} 
&= 
2\Gamma(f, \partial_t^+ Q_t^W f|_{t=0})
\\&=
2\Gamma(f, \partial_t^+ P_t^W f|_{t=0})\\
&=\partial_t^+ \Gamma P_t^W f|_{t=0}
\end{align*}
giving $\partial_t^+ \Gamma P_t^W f|_{t=0} \leq 2 \Gamma(f,\Delta f)$ which proves $(\ref{i:dtGammaPtW})$.

We next prove $(\ref{i:Gamma2PtW})$.
Since $f$ is constant on $\cl(W)$, we have $\Gamma f=0$ on $W$ and $\Gamma_2 f \geq 0$ on $W$ due to \cite[Proposition~1.1]{HuaLin16}. 

It is left to prove $(\ref{i:PinftyW})$.
Obviously, $\inf_V P_t^W f$ is non-decreasing, and $\sup_V P_t^W f$ is non-increasing in $t$. Therefore both limits exist when $t \to \infty$.
Let $\eps>0$ and let
$c:= \lim_{t\to \infty} \inf_V P_t^W f$ and $C:= \lim_{t\to \infty} \sup_V P_t^W f$. We aim to show $C=c$.
Let $\eps>0$ and let $T>0$ s.t. $|c- \inf_V P_T^W f| < \eps$.
Observe
\[
P_t f \leq Q_t^W f \leq P_t f + \inf_V Q_t^W f - \inf_V f.
\]
By applying iteratively and taking supremum, this estimate also holds for $P_t^W f$.
Applying to $g=P_T^W f$ yields
\[
P_{t} g \leq P_{t}^W g \leq P_{t} g +  \inf_V P_t^W g - \inf_V g  \leq P_{t} g +  c - \inf_V g \leq  P_{t} g +\eps.
\]
In particular, $\|P_t^W g - P_t g\|_\infty \leq \eps$.
Since $P_t g$ converges to a constant as $t \to \infty$, this gives $C-c \leq 2\eps$.
Taking $\eps \to 0$ proves claim $(\ref{i:PinftyW})$.
This finishes the proof of the theorem.
\end{proof}

\section{Gradient estimates}\label{sec:GradEstimates}
Having settled the machinery of the perpetual cutoff method, we can give a short proof of the gradient estimate requiring a curvature bound not everywhere.

\begin{theorem}\label{thm:gradEstimate}
Let $G$ be a graph with bounded weighted vertex degree.
Suppose $G$ satisfies $CD(K,\infty)$ on $V \setminus W$ for some $\emptyset \neq W \subset V$.
Let $f \in \ell_\infty^W$.
Then,
\[\Gamma P_t^W f \leq e^{-2Kt} P_t \Gamma f.
\]
\end{theorem}
\begin{proof}
Let $F(s):= e^{2Ks} P_{t-s}(\Gamma P_{s}^W f).$
Due to Lemma~\ref{lem:DiniPtGamma} $(\ref{i:dtPtU})$, we can use product and chain rule for calculating the right derivative of $F$. Thus,
\begin{align*}
\partial_s^+ F(s) &= 2KF(s) + e^{2Ks}P_{t-s}(\partial_s^+ - \Delta) \Gamma P_{s}^W f \\ 
&\leq 2KF(s) -  2e^{2Ks} P_{t-s} \Gamma_2 P_{s}^W f
\end{align*}
where the inequality follows from Theorem~\ref{thm:PtWProperties} $(\ref{i:dtGammaPtW})$.
Observe $\Gamma_2 P_{t-s}^W f \geq K \Gamma P_{t-s}^W f$ on $V\setminus W$ due to $CD(K,\infty)$, and
$\Gamma_2 P_{t-s}^W f \geq 0 = K \Gamma P_{t-s}^W f$ on $W$ due to Theorem~\ref{thm:PtWProperties} $(\ref{i:Gamma2PtW})$.
Putting together yields
\[
\partial_s^+ F(s) \leq 2KF(s) - 2K e^{2Ks} P_{t-s} \Gamma P_{s}^W f = 0.
\]
Hence by Proposition~\ref{pro:Dini} $(\ref{i:Dininonnegative})$, we see that $F(s)$ is non-increasing which immediately implies the claim of the theorem.
\end{proof}

\section{Distance bounds}\label{sec:DistanceBounds}
Using the gradient estimates, we prove the distance bound analogically to the diameter bounds in \cite{liu2016bakry}.

\begin{theorem}\label{thm:DistanceBounds}
Let $G$ be a graph with bounded vertex degree $ \Deg \leq D$.
Suppose $G$ satisfies $CD(K,\infty)$ on $V \setminus W$ for some $K>0$ and some $\emptyset \neq W \subset V$. Then,
\[
d(\cdot, W) \leq \frac{2D}K + 1.
\]
\end{theorem}

\begin{proof}
First note that it suffices to prove $d(\cdot, \cl(W)) \leq 2D/K$ where $\cl(W)=\{v \in V: d(v,W) \leq 1\}$.
Let $f:=d(\cdot, \cl(W)) \wedge \frac {3D} K \in \ell_\infty^W(V)$.
Due to Theorem~\ref{thm:PtWProperties}  $(\ref{i:dtPtW})$, and due to $|\Delta g|^2 \leq 2D\Gamma g$ for all $g \in C(V)$,
 and due to Theorem~\ref{thm:gradEstimate}, we have
\begin{align*}
|\partial_t^+ P_t^W f|^2 \leq \|\Delta P_t^W f\|_\infty^2 \leq 2D \|\Gamma P_t^W f\|_\infty 
\leq 2D e^{-2Kt}\|\Gamma f\|_\infty \leq e^{-2Kt}D^2
\end{align*}
where the last estimate follows from $\Gamma g \leq D/2$ for every 1-Lipschitz function $g\in C(V)$.
Hence for every $x \in V$ and $w \in \cl(W)$,
\begin{align*}
f(x) &= f(x) - P_\infty^W f(x) + P_\infty^W f(w) - f(w) \\
&\leq \overline{\int_0^\infty} |\partial_t^+ P_t^W f(x)| dt +  \overline{\int_0^\infty} |\partial_t^+ P_t^W f(w)| dt \\
&\leq 2 \int_0^\infty e^{-Kt} D \\
&= \frac {2D}K
\end{align*}
where we applied Theorem~\ref{thm:PtWProperties} $(\ref{i:PinftyW})$ for the identity, and 
Proposition~\ref{pro:Dini} $(\ref{i:DiniIntegral})$ for the first estimate where 
the continuity of $P_t^W$ follows from Theorem~\ref{thm:PtWProperties} $(\ref{i:Continuity})$.
This immediately implies the claim of the theorem.
\end{proof}

\section{Discussion and outlook} 
\label{sec:Discussion}

The perpetual cutoff method seems to exhibit a vast versatility. It subtly intertwines with both Ollivier and Bakry-Emery curvature giving interesting new gradient estimates. There is good hope that it also works well both for exponential Bakry-Emery curvature and entropic curvature on graphs. It is conceivable that the perpetual cutoff method gives new insights on non-discrete settings like metric measure spaces, Markov processes or Dirichlet forms.
It seems natural to ask whether the gradient estimates and distance bounds can be extended to curvature bounds with finite dimension term.
We now point out a different direction for applying
the perpetual cutoff method on graphs.
Given a graph $G=(V,w,m)$ and an ordered sequence $V\supset W_1 \supset W_2 \supset \ldots$ s.t $V\setminus W_k$ is finite and $\bigcap W_k = \emptyset$. Under which conditions do we have $\lim_{k\to \infty} P_t^{W_k} f = P_t f$?
We remark that the pointwise limit exists since $P_t^{W_k} f$ is non-increasing in $k$.
A positive answer to this question would make the perpetual cutoff method compatible with the maximum principle which expectedly relies on finiteness of $V \setminus W$.

\section*{Acknowledgments}
The author is grateful for financial support by the German National Merit Foundation and by the MPI MiS Leipzig. He wants to thank Christian Rose, Norbert Peyerimhoff and Shiping Liu for inspiring and encouraging the work on non-constant Bakry-Emery curvature. He also wants to thank Gabor Lippner and Mark Kempton for useful discussions.


\printbibliography 
\textcolor{white}{}\\  
Florentin M\"unch,\\
Max Planck Institute for Mathematics in the Sciences Leipzig, Germany\\
\texttt{muench@mis.mpg.de}\\
\end{document}